\documentclass[12pt]{amsart}
\usepackage{color,graphicx,array, amssymb, amscd}
\usepackage[mathscr]{eucal}
\usepackage{amsmath}
\usepackage{amssymb}
\usepackage{amscd}
\usepackage{latexsym}
\usepackage{a4wide}
\usepackage{multirow}
\usepackage{color}
\theoremstyle{plain}

\newtheorem{Theorem}{Theorem}[section]

\newtheorem{Proposition}[Theorem]{Proposition}
\newtheorem{Corollary}[Theorem]{Corollary}
\theoremstyle{Definition}
\newtheorem{Definition}{Definition}[section]
\theoremstyle{Remark}

\newtheorem{Remark}[Theorem]{Remark} 
 
\numberwithin{equation}{section}
\def\Vec#1{\mbox{\boldmath $#1$}}

\newcommand{\SL}{{\rm SL}_4 \mathbb R}

\newcommand{\gs}{g_2}
\newcommand{\gf}{g_1}
\newcommand{\js}{J_2}
\newcommand{\jf}{J_1}
\newcommand{\id}{\operatorname{id}}

\newcommand{\di}{\operatorname{diag}}
\newcommand{\offdi}{\operatorname{offdiag}}

\makeatletter 
\def\iddots{\mathinner{\mkern1mu\raise\p@
    \hbox{.}\mkern2mu\raise4\p@\hbox{.}\mkern2mu
    \raise7\p@\vbox{\kern7\p@\hbox{.}}\mkern1mu}}
\def\adots{\mathinner{\mkern2mu\raise\p@\hbox{.} 
 \mkern2mu\raise4\p@\hbox{.}\mkern1mu
 \raise7\p@\vbox{\kern7\p@\hbox{.}}\mkern1mu}}
\makeatother
\begin{document}
\title{The Gauss maps of Demoulin surfaces with conformal coordinates}

\author[J.~Inoguchi]{Jun-ichi Inoguchi}
\address{
Institute of Mathematics, 
University of Tsukuba, 
Tsukuba, 305-8571, Japan}
\email{inoguchi@math.tsukuba.ac.jp}
\thanks{The first named author is partially supported by JSPS Kakenhi 
Grant Number JP19K03461}
\author[S.-P.~Kobayashi]{Shimpei Kobayashi}
\address{Department of Mathematics, 
Hokkaido University, Sapporo, 
060-0810, Japan}
\email{shimpei@math.sci.hokudai.ac.jp}
\thanks{The second named author is partially supported by JSPS Kakenhi 
Grant Number JP18K03265}

\dedicatory{In Memory of Professor Zhengguo Bai {\rm(1916}-{\rm2015)}}


\begin{abstract}
Demoulin surfaces in real projective $3$-space are investigated.
Our result enable us to establish a 
generalized Weierstrass type representation for 
definite Demoulin surfaces by virtue of 
primitive maps into a certain
semi-Riemannian $6$-symmetric space.
\end{abstract}
\maketitle

\section*{Introduction}Professor Zhengguo Bai have done great contributions 
in projective differential geometry. 
For example, he solved the so-called Fubini's problem~\cite{Pa1956} 
(\textit{cf} \cite{Su1983}).

Projective differential geometry of surfaces is a treasure box 
of infinite dimensional integrable systems.
For instance, 
harmonic maps of Riemann surfaces 
into complex projective space  $\mathbb{CP}^n$ 
(the $\mathbb{CP}^n$-sigma models in particle 
physics) are typical examples of $2$-dimensional 
integrable systems. One of the key clue of the 
study of harmonic maps into complex projective space
is the use of \emph{harmonic sequences} 
introduced by Chern and Wolfson~\cite{CW}. 
It should be emphasized that the basic 
idea of harmonic sequence goes back to 
\emph{Laplace sequence} in classical projective differential geometry, see~\cite{BW1992}. 

From modern point of view, the Laplace sequence produces 
$2$-dimensional Toda field equation of type $\mathrm{A}_{\infty}$, 
see~\cite{Darboux, Sasaki}.
In particular, the periodic Laplace sequence produces 
$2$-dimensional periodic Toda field equations. 
For example, Laplace sequences of period $2$ 
produce sinh-Gordon equation. 
\emph{\c{T}i\c{t}eica equation} is 
obtained as Laplace sequence of period $3$, and it
 is a structure equation of affine spheres~\cite{DFKW}. 
Laplace sequences of period $4$ were studied by Su~\cite{Su1936, Su1964}. 
Hu gave a Darboux matrix, that is,
the simple type dressing for such a sequence~\cite{Hu}.

This article addresses Laplace sequences of period $6$. 
The Toda field equation derived from those sequences   
is a structure equation of \emph{Demoulin surfaces} in 
real projective $3$-space $\mathbb{R}\mathbb{P}^3$. 

Godeaux gave a method for studying projective surfaces
through their Pl{\"u}ker images in real projective $5$-space $\mathbb{R}\mathbb{P}^5$.
His method relies on the consideration of the Laplace 
sequence associated with the Pl{\"u}ker image, called the 
\emph{Godeaux sequence}. For a characterization of Demoulin surfaces
in terms of Godeaux sequences, see~\cite[\S 4.8]{Sasaki}.
Bai~\cite{Pa} studied Godeaux sequences of quadrics.
 
In~\cite{Kobayashi2015}, the second named author 
considered two Gauss maps of surfaces in $\mathbb{R}\mathbb{P}^3$
with indefinite \emph{projective metric} and characterized
projective minimal surfaces and Demoulin surfaces in terms of
harmonicities of the Gauss maps.
In this paper, we consider those surfaces with \emph{positive definite} projective metric. This paper is organized as follows:
After preparing prerequisite knowledge on projective surface theory in Sections \ref{sc:surfacetheory}-\ref{sc:PMD}, 
we parametrize the space of all conformal $2$-spheres in 
$\mathbb{R}\mathbb{P}^3$ in Section \ref{sec:Plucker}. 
We will show that 
the space of all conformal $2$-spheres is realized as a 
semi-Riemannian symmetric space. The Gauss maps introduced in 
this paper take values in this symmetric space.  
In Section \ref{sc:1stGauss}, we introduce the first-order Gauss map for 
a surface in $\mathbb{R}\mathbb{P}^3$ as a congruence of conformal 
$2$-spheres which has the first-order contact to the surface. Definite 
Demoulin surfaces are characterized as surfaces with conformal first order Gauss map.
In addition definite Demoulin surfaces and definite projective minimal coincidence surfaces 
are characterized by harmonicity of first order Gauss map. 
In the final section we will show that every definite Demoulin surface can be 
constructed by a primitive map into certain semi-Riemannian $6$-symmetric space fibered over the semi-Riemannian symmetric space
of all conformal $2$-spheres. 

\medskip  
  
Throughout this paper, we use the following abbreviation:
\[
\mathrm{diag} (a_1, a_2, \cdots, a_n) 
= \begin{pmatrix} 
 a_1  &&& \\& a_2 &&\\ &&\ddots&\\  &&&a_n 
 \end{pmatrix}, \quad
\mathrm{offdiag} (a_1, a_2, \cdots, a_n) = \begin{pmatrix} 
  &&&a_1 \\&& a_2&\\ & \iddots & &\\  a_n &&& 
 \end{pmatrix}.
\]
  
\section{Projective surface theory}\label{sc:surfacetheory}

Let $\mathfrak{f}:M\to \mathbb{R}\mathbb{P}^3$ be an immersed surface in the real projective 
$3$-space $\mathbb{R}\mathbb{P}^3$. Take a
simply connected region $\mathbb{D} \subset M$ and homogeneous 
coordinate vector field $f=(f^0,f^1,f^2,f^3):
\mathbb{D}\to\mathbb{R}^{4}\setminus\{\Vec{0}\}$. 
Let $D$ be the natural affine connection on $\mathbb{R}^4$ and 
$\Omega$ a volume element so that $D\Omega=0$. Thus 
$(\mathbb{R}^4,D,\Omega)$ is an \emph{equiaffine $4$-space}. 
One can take a vector field $\xi$ transversal to both 
$f$ and the radial vector field $\zeta=\sum_{i=0}^{3}
x^{i}\partial/\partial{x^i}$. 
Then $\xi$ induces an affine connection $\nabla$ on 
$\mathbb{D}$ and symmetric tensor fields $h$ and $T$ via the 
\emph{Gauss formula}:
\[
D_{X}f_{*}Y=f_{*}(\nabla_{X}Y)+h(X,Y)\xi+
T(X,Y)\zeta,
\quad 
X,Y\in\varGamma(T\mathbb{D}).
\]
Moreover we have the following \emph{Weingarten formula}:
\[
D_{X}\xi=-f_{*}(SX)+\tau(X)\xi+\rho(X)\zeta.
\]
The triplet $(\mathbb{D},f,\xi)$ is a 
\emph{centroaffine surface} (of codimension $2$) in $\mathbb{R}^4$ in the sense of~\cite{NS, NSBook}. 
We introduce an area element $\vartheta$ on $\mathbb{D}$ by 
\[
\vartheta(X,Y)=\Omega(f_{*}X,f_{*}Y,\xi,\zeta).
\]%
The \emph{cubic form} $C$ is defined by
\[
C=\nabla h+\tau\otimes h.
\]

The non-degeneracy of $h$ is independent of the choice of 
$\xi$. In addition, the conformal class $[h]$ of $h$ is independent of 
$\xi$. Thus the property 
``$h$ is \emph{positive definite}" is well defined for $f$. 
Throughout this article, we assume that $h$ is positive definite.

When we take $\xi$ so that $\tau=0$, then 
$(\mathbb{D},f,\xi)$ is said to be \emph{equiaffine}.
An equiaffine centroaffine immersion $f$ is said to be 
\emph{Blaschke} if $\vartheta$ coincides with the area element 
of the metric $h$. 

On the other hand Nomizu and Sasaki~\cite{NS} showed that there exits a 
transversal vector field $\xi$ such that 
\begin{equation}\label{eq:NS-normalization}
\mathrm{tr}_{h}T+\mathrm{tr}\>S=0.
\end{equation}
Such a vector field is called a \emph{pre-normalized transversal vector field}.
In particular, pre-normalized transversal vector field $\xi$ so that 
$(\mathbb{D},f,\xi)$ is a Blaschke immersion is unique up to sign. 
In such a choice, the pair
surface $(f,\xi)$ is called a \emph{pre-normalized Blaschke immersion}.

Let us take another homogeneous coordinate vector field 
$\tilde{f}=\phi f$. Here $\phi$ is a smooth (nonzero) function.
Then the connection $\tilde{\nabla}$ induced from 
$\tilde{f}$ is projectively equivalent to $\nabla$. 
The equiaffine property is preserved under the change 
$f$ by $\phi f$. 

Let us denote by $\nabla^h$ the Levi-Civita connection of $h$. Then 
the scalar field $J=h(K,K)/2$ is called the 
\emph{Fubini-Pick invariant} of $f$. Here $K=\nabla-\nabla^h$.
The Riemannian metric $Jh$ is projectively invariant and called the 
\emph{projective metric} of $\mathfrak f$. 
Although $C$ itself is \textit{not} projective invariant, 
its conformal class is projective invariant (see~\cite{MN}). 
When $(f,\xi)$ is pre-normalized Blaschke, the projective metric
is given by $h(\nabla h,\nabla h)h/8$.

For more details on centroaffine immersions and projective immersions, we refer 
to~\cite{NS, NSBook}.

\section{Wilczynski frames}\label{sc:Wilczynski}
Let $\mathfrak{f}:M\to \mathbb{R}\mathbb{P}^3$ be an immersed surface with 
\emph{positive definite projective metric}.
We regard $M$ as a Riemann surface with respect to the 
conformal structure $[Jh]$ determined by the projective metric $Jh$. 

We take a simply connected 
complex coordinate region $\mathbb{D}$ with 
coordinate $z = x + i y$ on $\mathbb{D}$ and a  
lift $f=(f^0, f^1, f^2, f^3): \mathbb{D}\to \mathbb{R}^4 \setminus\{\Vec{0}\}$.
Then 
the \emph{canonical system} of $\mathfrak f$ 
is given by
\begin{equation}\label{eq:1.1}
f_{zz}=bf_{\bar z}+pf,\ \ 
f_{\bar{z}\bar{z}}=\bar{b}f_{z}+\bar{p}f
\end{equation}
for some smooth functions $b$ and $p$, see~\cite[p.~121]{Sasaki}. 
Note that the subscripts $z$ and $\bar{z}$ denote the 
 partial derivative of $z$ and $\bar{z}$, respectively:
 \begin{equation*}
 \frac{\partial}{\partial z} := \frac12 \left( \frac{\partial}{\partial x}-\sqrt{-1} \frac{\partial}{\partial y}\right), 
\quad 
 \frac{\partial}{\partial \bar z} := \frac12 \left( \frac{\partial}{\partial x}+\sqrt{-1} \frac{\partial}{\partial y}\right).
 \end{equation*}
Assume that $f^{0}\not=0$, then $\mathfrak f$ is given by the 
inhomogeneous coordinate 
$\mathfrak f=(f^1,f^2,f^3)/f^0$. 
The canonical system is 
rewritten as
\begin{equation}\label{eq:canonicalsystem}
\mathfrak f_{zz}=b\mathfrak f_{\bar z}-2(\log f^0)_{z}\mathfrak f_{z},
\quad 
\mathfrak f_{\bar{z}\bar{z}}=\bar{b}\mathfrak f_{z}-2(\log f^0)_{\bar{z}}
\mathfrak f_{\bar{z}}.
\end{equation}
The integrability condition of the canonical system is 
(\textit{cf.} \cite[\S 2.3]{Sasaki}): 
\begin{gather*}
p_{\bar z}=b\bar{b}_{z}+\frac{1}{2}b_{z}\bar{b}-\frac{1}{2}b_{\bar{z}\bar{z}}, \\
\mathrm{Im}\>(b_{\bar{z}\bar{z}\bar{z}}-b\bar{b}_{z\bar{z}}
-2b\bar{p}_{\bar z}-2b_{\bar z}\bar{b}_{z}-4b_{\bar z}\bar{p})=0.
\end{gather*}

The Fubini-Pick invariant is given by $J=8|b|^2$ and hence the 
projective metric is 
$8|b|^2\,dzd\bar{z}$. The cubic form of $\mathfrak f$ is  
given by  
$C=-2(b\,dz^3+\bar{b}d\bar{z}^3)$ (see~\cite[p.~54, Definition, \S 4.8]{SasakiRokko}). Note that when $f$ is pre-normalized Blashcke, 
then the projective metric is expressed as $2|b|^2\,dzd\bar{z}$.

Hereafter we assume that $b\not=0$. Note that 
when $C=0$, $\mathfrak f$ is a part of a quadratic surface 
(Wilczynski~\cite{Wil1}, Pick~\cite{Pick}. See also~\cite[Theorem 4.4]{SasakiRokko}.)

The \emph{Wilczynski frame} $F$ of $\mathfrak f$ is defined by 
\[
F=(f,f_1,f_2,\eta),
\]
where
\[
f_{1}:=f_{z}-\frac{\>\>\bar{b}_z\>\>}{2\bar{b}}f,
\quad 
f_{2}:=f_{\bar z}-\frac{b_{\bar z}}{2b}f,
\quad 
\eta=f_{z\bar{z}}-\frac{\bar{b}_z}{2\bar{b}}f_{\bar{z}}
-\frac{b_{\bar z}}{2b}f_{z}+
\left(
\frac{|b_{z}|^2}{4|b|^2}-\frac{|b|^2}{2}
\right)f.
\]%
Then a straightforward computation shows that 
 the Wilczynski frame $F$ satisfies the following equations:
\begin{equation}\label{eq:movingframe}
F_{z}=FU \;\;\mbox{and} \;\;F_{\bar z}=FV,
\end{equation}
where 
\begin{align*}
U=&\left(
\begin{array}{cccc}
\bar{b}_z/(2\bar{b})& P & k & b\bar{P}\\
1 & -{\bar b}_z/(2\bar{b}) & 0 & k\\
0 & b & \bar{b}_z/(2\bar{b}) & P\\
0 & 0 & 1 & -\bar{b}_z/(2\bar{b})
\end{array}
\right),
\\
\smallskip
\\
V=&\left(
\begin{array}{cccc}
b_{\bar z}/(2b) & \bar{k} & \bar{P} & \bar{b}P\\
0 & b_{\bar z}/(2b) & \bar{b} & \bar{P}\\
1 & 0 & -b_{\bar z}/(2b) & \bar{k}\\
0 & 1 & 0  & -b_{\bar z}/(2b)
\end{array}
\right).
\\
\smallskip
\end{align*}
Here we introduced functions $k$ and $P$ of as follows:
\begin{gather}
 k = \frac{|b|^2  - (\log b)_{z\bar{z}}}{2},\;\;\; 
\label{eq:kell}\\
 P = p + \frac{b_{\bar z}}{2} -\frac{\bar{b}_{zz}}{2 \bar{b}} +
 \frac{\bar{b}_{z}^2}{4 \bar{b}^2}.
 \label{eq:QandP}
\end{gather}
 The compatibility conditions of \eqref{eq:movingframe} are 
 \begin{gather}
  \bar{P}_{z} = k_{\bar z} + k \frac{b_{\bar z}}{b}, \;\; 
 \label{eq:comp1}\\
\mathrm{Im}\>( \bar{b} P_{z}+ 2 \bar{b}_{z} P)=0. \label{eq:comp2}
 \end{gather}
 These equations are nothing but the 
 \emph{projective Gauss-Codazzi equations}
 of a surface $\mathfrak{f}$. 
 One can see that $Pdz^2$ and $2b^2\bar{P}dz^4$ are 
 globally defined on $M$ and projectively invariant~\cite{Ferapontov}.
 
 Since both $U$ and $V$ are trace free, 
 the Wilczynski frame 
 $F$ takes values in $\mathrm{SL}_{4}\mathbb{C}$ up 
 to initial condition. 
 Moreover, if we choose at some 
  base point $z_* \in \mathbb D$ and $F(z_*)= \id$, then
 the frame $F$ takes values in $\mathrm{SL}_{4}\mathbb{R}$ by conjugation of a simple complex matrix:
 \begin{equation}\label{eq:L}
 \operatorname{Ad}(L)F  \in \SL, \quad  L = 
 \frac1 {\sqrt{2}}
  \begin{pmatrix} 
  \sqrt{2} & 0 & 0 &0 \\ 
  0 & \sqrt{-1}  &  -\sqrt{-1} &0 \\ 
  0 & 1 &  1  &0 \\ 
  0 & 0 & 0 & \sqrt{2} 
  \end{pmatrix}.
 \end{equation}

\section{Projective minimal surfaces and definite Demoulin surfaces}\label{sc:PMD}
A surface $\mathfrak{f}:M\to\mathbb{R}\mathbb{P}^3$ with 
positive definite projective metric is said to be a 
\emph{projective minimal surface} if it is a 
critical point of the area functional of the projective 
metric (called the \emph{projective area functional}):
Then the projective minimality 
 can be computed as in~\cite{Th}:
 \begin{equation}\label{eq:projmin}
 \bar{b} P_{z} + 2 \bar{b}_{z} P =0.
 \end{equation}
 where the functions $P$ is defined in \eqref{eq:QandP}.
It should be remarked that the projective minimality 
\eqref{eq:projmin} implies the second equation 
\eqref{eq:comp2} of the projective Gauss-Codazzi equations.
There is a particular class of projective minimal surfaces with 
positive definite projective metric. 

A surface with positive definite 
projective metric is said 
to be a \emph{definite Demoulin surface} 
if it satisfies $P=0$. 
The Demoulin property is originated from 
Demoulin transformations of surfaces in 
$\mathbb{R}\mathbb{P}^3$. For more details, we refer 
to~\cite{Sasaki}.

\section{The Pl{\"u}cker quadric and the space of conformal spheres}\label{sec:Plucker}
\subsection{The Pl{\"u}cker quadric}
Take a volume element $\Omega$ on $\mathbb{R}^4$ 
parallel with respect to the natural affine connection $D$.
Then we can introduce a 
scalar product $\langle\cdot,\cdot\rangle$ on 
$\wedge^{2}\,\mathbb{R}^4$ by
\[
\langle\alpha,\beta\rangle
=\Omega(\alpha\wedge \beta)
,\quad 
\alpha,
\beta\in \wedge^{2}\mathbb{R}^4.
\]
One can check that $\langle\cdot,\cdot\rangle$ is of signature $(3,3)$.
In fact, let $\{e_0,e_1,e_2,e_3\}$ be the natural basis of 
$\mathbb{R}^4$. Denote by
$\{e^0,e^1,e^2,e^3\}$
the dual basis of $\{e_0,e_1,e_2,e_3\}$. 
Then with respect to the volume element $\Omega=e^0\wedge e^1 \wedge e^2\wedge e^3$, 
the basis 
$\{
e_0\wedge e_1,e_0\wedge e_2,e_0\wedge e_3,
e_1\wedge e_2,e_3\wedge e_1,
e_2\wedge e_3
\}$ of $\wedge^2\mathbb{R}^4$, the scalar 
product $\langle\cdot,\cdot\rangle$ 
is determined by the matrix
$\mathrm{offdiag}(1,1,1,1,1,1)$. 
The special linear group $\mathrm{SL}_4\mathbb{R}$ acts 
on $\wedge^{2}\mathbb{R}^{4}$ via the action:
\[
\mathrm{SL}_{4}\mathbb{R}\times
\wedge^{2}\mathbb{R}^{4}\to
\wedge^{2}\mathbb{R}^{4};\quad 
(g,v\wedge w)\longmapsto gv\wedge gw.
\]
One can see that this action is isometric with 
respect to $\langle\cdot,\cdot\rangle$. This fact 
implies the Lie group isomorphism 
$\mathrm{PSL}_{4}\mathbb{R}\cong\mathrm{SO}^{+}_{3,3}$.
Here $\mathrm{SO}^{+}_{3,3}$ denotes the identity component 
of the semi-orthogonal group $\mathrm{O}_{3,3}$.

Next we consider the Pl{\"u}cker 
embedding of the Grassmannian manifold 
$\mathrm{Gr}_{2}(\mathbb{R}^4)$ of $2$-planes 
in $\mathbb{R}^4$ into the projective 
$5$-space $\mathbb{R}\mathbb{P}^5=\mathbb{P}(\wedge^{2}\mathbb{R}^4)$.
The Pl{\"u}cker coordinates of the $2$-plane spanned by 
$(a^0,a^1,a^2,a^3)$ and 
$(b^0,b^1,b^2,b^3)$ is $[p_{01}:p_{02}:p_{03}:p_{23}:p_{31}:p_{12}]$, 
where 
\begin{equation}\label{eq:PlCoord}
p_{ij}=\det
\left(
\begin{array}{cc}
a^{i} & b^{i}\\
a^{j} & b^{j}
\end{array}
\right).
\end{equation}
The Pl{\"u}cker coordinates 
$[p_{01}:p_{02}:p_{03}:p_{23}:p_{31}:p_{12}]$ of 
$a\wedge b$ satisfies the 
\emph{quadratic Pl{\"u}cker relation}:
\begin{equation}\label{eq:PlRel}
p_{01}p_{23}+p_{02}p_{31}+p_{03}p_{12}=0.
\end{equation}
Thus the Pl{\"u}cker image of $\mathrm{Gr}_{2}(\mathbb{R}^4)$ 
is a projective variety 
(called the \emph{Pl{\"u}cker quadric}) of 
$\mathbb{R}\mathbb{P}^5$ determined by the 
equation \eqref{eq:PlRel}. 
Moreover the  Pl{\"u}cker relation means that 
\newline
\noindent 
$(p_{01},p_{02},p_{03},p_{23},p_{31},p_{12})$ is null with respect to 
$\langle\cdot,\cdot\rangle$. Namely the 
Pl{\"u}cker image of $\mathrm{Gr}_{2}(\mathbb{R}^4)$ is 
the projective light cone $\mathbb{P}(\mathcal{L})$ of $\wedge^{2}\mathbb{R}^4=\mathbb{R}^{3,3}$.

Now let us consider a line $\ell$ in $\mathbb{R}\mathbb{P}^3$ 
connecting two points 
$a=[a^0:a^1:a^2:a^3]$ and 
$b=[b^0:b^1:b^2:b^3]$. 
The Pl{\"u}cker image of $\ell$ in 
$\mathbb{R}\mathbb{P}^5=\mathbb{P}(\wedge^{2}\mathbb{R}^4)$ is 
\[
a\wedge b=[p_{01}:p_{02}:p_{03}:p_{23}:p_{31}:p_{12}]
\]
with \eqref{eq:PlCoord}.
Hence the space $\mathcal{P}$ of lines in $\mathbb{R}\mathbb{P}^3$ is 
identified with the Pl{\"u}cker quadric. 
This identification is called the \emph{Klein correspondence}.

\begin{Remark}{\rm 
The conformal compactification 
of semi-Euclidean $4$-space 
$\mathbb{R}^{2,2}$ of neutral signature 
is obtained as the projective light cone 
$\mathbb{P}(\mathcal{L})\subset\mathbb{R}\mathbb{P}^{5}$ equipped with the conformal structure induced from 
$\mathbb{R}^{3,3}$. The action of $\mathrm{PSL}_{4}\mathbb{R}\cong\mathrm{SO}^{+}_{3,3}$ on 
$\mathbb{P}(\mathcal{L})$ is conformal. One can see that the Pl{\"u}cker quadric $\mathcal{P}=\mathbb{P}(\mathcal{L})$ is 
isomorphic to $\mathrm{Gr}_{2}(\mathbb{R}^4)\cong (\mathbb{S}^2\times\mathbb{S}^2)/\mathbb{Z}_2$ 
(equipped with the standard conformal structure of neutral signature) as a conformal manifold. 
Note that on $\mathbb{P}(\mathcal{L})$, there exits a complex structure compatible to the 
standard neutral metric. The standard neutral metric is neutral K{\"a}hler 
with respect to the complex structure. In particular, the K{\"a}hler form is 
regarded as a standard
symplectic form on $\mathbb{P}(\mathcal{L})$. For more information on 
conformal geometry of $\mathbb{P}(\mathcal{L})$, see~\cite{LM}.
}
\end{Remark}

\subsection{The space of conformal spheres}
A quadric in $\mathbb{R}\mathbb{P}^3$ is 
a surface of the form $\{[v]\in\mathbb{R}\mathbb{P}^3\>|\>
q(v,v)=0\}$, 
where $q$ is a scalar product of $\mathbb{R}^4$.
For our purpose we choose a Lorentzian scalar product 
$q=\langle\cdot,\cdot\rangle$ on $\mathbb{R}^4$ and regarded it as a 
Minkowski $4$-space $\mathbb{R}^{1,3}$. Then 
the quadric is nothing but the 
conformal $2$-sphere (Riemann sphere) in $\mathbb{R}\mathbb{P}^3$.  
The space of conformal $2$-spheres in
$\mathbb{R}\mathbb{P}^3$ is parametrized as
the space $\mathcal{Q}$ of $4\times 4$ symmetric matrices with 
determinant one and signature $(1,3)$. 
In fact, the conformal $2$-sphere 
is given by the Lorentzian scalar product 
$q(u,v)=uQv^{\mathrm T}$ with $Q\in\mathcal{Q}$.

The special linear group 
$\mathrm{SL}_{4}\mathbb{R}$ acts 
transitively on $\mathcal{Q}$ via the action 
$(g,Q)\longmapsto gQg^{\mathrm T}$ with $g\in\mathrm{SL}_{4}
\mathbb{R}$ and $Q\in\mathcal{Q}$.
The stabilizer at

\begin{equation}\label{eq:hatL}
    \hat \jf=  
 \begin{pmatrix} 
  0 & 0 & 0 &1 \\ 
  0 & 1 &  0 &0 \\ 
  0 & 0 & 1 &0 \\ 
  1 & 0 & 0 &0
  \end{pmatrix},
\end{equation}%
is given by 
$\hat{K}_1 = \{ a \in \SL \;|\; a \hat \jf 
 a^{\mathrm T} = \hat \jf\}$, which 
 is isomorphic to the identity component 
 $\mathrm{SO}^{+}_{1, 3}$ of the semi-orthogonal group 
 $\mathrm{O}^{+}_{1, 3}$ of signature $(1, 3)$. 
 Thus $\mathcal{Q}$ is isomorphic to the homogeneous  space $\mathrm{SL}_{4}\mathbb{R}
 /\mathrm{SO}^{+}_{1,3}\cong \mathrm{SO}^{+}_{3,3}/\mathrm{SO}^{+}_{1,3}$.

We introduce a scalar product $\langle\cdot,\cdot\rangle$ at 
$Q\in\mathcal{Q}$ by 
\[
\langle X,Y\rangle_{Q}=\mathrm{Tr}\>(Q^{-1}X\,Q^{-1}Y),
\quad X,Y\in T_{Q}\mathcal{Q}.
\]
Note that at the origin of 
$\mathrm{SO}^{+}_{3,3}/\mathrm{SO}^{+}_{1,3}$,
and $8\langle\cdot,\cdot\rangle$ is the Killing form of 
$\mathfrak{sl}_{4}\mathbb{R}$.
This scalar product is invariant under the action of 
$\mathrm{SL}_{4}\mathbb{R}$. 
In fact, 
\[
\langle gXg^{\mathrm T},gYg^{\mathrm T}\rangle_{gQg^{\mathrm T}}
 =\mathrm{Tr}\,  ((gQ g^{\mathrm T})^{-1} gXg^{\mathrm T} (gQ g^{\mathrm T})^{-1} gY g^{\mathrm T}) 
 =  \langle X, Y\rangle_{Q}.
\]
Thus $\mathcal{Q}=\mathrm{SL}_{4}\mathbb{R}/\hat{K}_1$ is a 
semi-Riemannian symmetric space corresponding to the outer 
involution
\[
\hat{\tau}_{1}(X)=\hat{J}_{1}(X^{\mathrm T})^{-1}\hat{J}_{1}.
\]
\begin{Remark}
{\rm 
The space of lines in the Pl{\"u}cker quadric $\mathcal{P}$ is identified with the Grassmannian manifold 
of all null $2$-planes in $\mathbb{R}^{3,3}$:
\[
\mathcal{Z}=\{
W\in \mathrm{Gr}_{2}(\mathbb{R}^{3,3})
\>\>|\>\>W \>\>\mbox{is a null 2-plane in} \>\> \mathbb{R}^{3,3}\}
\cong 
\mathrm{SO}^{+}_{3,3}/\mathrm{SO}^{+}_{2, 2}.
\]
For surfaces in $\mathbb{R}\mathbb{P}^3$ with 
\emph{indefinite} projective metric, 
two kinds of Gauss maps are considered in 
our previous work~\cite{Kobayashi2015}. 
Those Gauss maps take value in 
the space of quadrics determined by 
scalar products of signature $(2,2)$ of $\mathbb{R}^4$. 
The space of all quadrics derived from 
such scalar products are identified with 
the semi-Riemannian symmetric space 
$\mathrm{SO}^{+}_{3,3}/\mathrm{SO}^{+}_{2, 2}$.
}
\end{Remark}

\section{Demoulin surfaces and the first order Gauss maps}
\label{sc:1stGauss}
 In this section, we define the first-order Gauss map for a surface 
 in $\mathbb R \mathbb P^3$.
 
\subsection{First-order Gauss map}\label{subsc:Gaussmap}
 Let $\mathfrak f: M \to \mathbb{R}\mathbb{P}^3$ be a surface and $F$ the corresponding 
 Wilczynski frame defined in \eqref{eq:movingframe} with 
 a base point $z_* \in \mathbb D$ and $F(z_*)=\id$.
 Let $L$ be the matrix defined in \eqref{eq:L} and 
 $\hat F$ be the $\SL$ matrix such that 
 \begin{equation*}
\operatorname{Ad}(L)F = \hat F.     
 \end{equation*}
 We now  define the \emph{first order Gauss map} $\gf$ as follows:
 \begin{equation}\label{eq:firstorder}
 \gf =   \hat F  \hat J_1 \hat F^{\mathrm T}=\mathrm{Ad}(L)(F \jf F^{\mathrm T}), 
  \end{equation}
 where the matrix $\hat{\jf}$ is the one 
given by \eqref{eq:hatL} and $J_1=
\mathrm{offdiag}(1,1,1,1)$. Note that $\mathrm{Ad}(L) J_1 = \hat J_1$.
Therefore the map $\gf$ takes values in the space of 
conformal $2$-spheres:
 \begin{equation}\label{eq:Gaussmap1}
 \gf : M \to \mathcal{Q} \cong \mathrm{SL}_{4}
\mathbb{R}/\hat{K}_{1} = \mathrm{SL}_{4}\mathbb{R}/\mathrm{SO}^{+}_{1,3}.
 \end{equation}
 This map $\gf$ is known to be a quadric which has 
 the first order contact to the surface and it does not have the
 second order contact, see~\cite[Section 22]{Lane}. 

 We now characterize the Demoulin surface by the first-order 
 Gauss map.
\begin{Proposition}\label{prop:conformality}
 The first-order Gauss map $\gf$ of a surface 
 $\mathfrak f$ in $\mathbb{R}\mathbb{P}^3$ with 
 positive definite projective metric 
 is conformal if and only if
 $\mathfrak f$ is a definite Demoulin surface.
\end{Proposition}
\begin{proof}
 A direct computation shows that 
\[
 \partial_z {\gf}= 2 (LF)
 \begin{pmatrix}
 b \bar P & k & P & 0 \\ 
 k & 0 & 0 &1 \\
 P & 0 & b & 0 \\
 0 & 1 & 0 & 0
 \end{pmatrix}
 (LF)^{\mathrm T},
\quad 
\partial_{\bar z} {\gf}= 2 (LF)
 \begin{pmatrix}
 \bar b P & \bar P & \bar k& 0 \\ 
 \bar P & \bar b  & 0 &0 \\
 \bar k & 0 & 0 & 1 \\
 0 & 0 & 1 & 0
 \end{pmatrix}(LF)^{\mathrm T}.
\]
 Thus 
\[
\langle \partial_z {\gf}, \partial_z {\gf} \rangle = 16 P, \; 
\langle  \partial_{\bar z}{\gf},  \partial_{\bar z}{\gf} \rangle  = 16 \bar P
\;\; \mbox{and} \;\;
\langle  \partial_z{\gf},  \partial_{\bar z}{\gf} \rangle = \langle \partial_{\bar z}{\gf},\partial_{z}{\gf} \rangle = 
 8(k + \bar k)+ 4 |b|^2.
\]
 Since  the coordinates $(z, \bar z)$  are null for the conformal 
 structure induced by $\mathfrak f$,
 the first-order Gauss map $\gf$ is conformal if and only if $P =0$.
\end{proof}

\subsection{Demoulin surfaces and projective minimal coincidence surfaces}\label{subsection6.2}
We set 
\[
G=\mathrm{Ad}(L^{-1})\mathrm{SL}_{4}\mathbb{R}
=\{L^{-1}XL\>|\>X\in\mathrm{SL}_{4}\mathbb{R}\}
\subset\mathrm{SL}_{4}\mathbb{C},
\]
where $L$ is defined in \eqref{eq:L}. 
The closed subgroup $G$ is a real form of $\mathrm{SL}_{4}\mathbb{C}$ which is 
isomorphic to $\mathrm{SL}_{4}\mathbb{R}$. 
The space 
$\mathcal{Q}$ of conformal $2$-spheres is isomorphic to 
$G/K_1$, where $K_1$ is 
\[
K_1=\{a\in G\>|\>aJ_{1}a^{\mathrm T}=J_1\}.
\]
 Let $\tau_1$ be the outer involution on 
the $G$ 
associated to $G/K_1$ given by
\[
\tau_1 (a) = \jf \left(a^{\mathrm T}\right)^{-1} \jf,
\quad a\in G.
\]

By abuse of notation, we denote the 
differential of $\tau_1$ by the same letter $\tau_1$:  
\begin{equation}\label{eq:tau2}
\tau_1 (X) = - \jf X^{\mathrm T}\jf,
\quad X\in\mathfrak{g}.
\end{equation}%
 Let us consider the eigenspace decomposition
 of $\mathfrak{g}$ with respect to $\tau_1$, that 
 is, $\mathfrak{g} = \mathfrak k_1 \oplus \mathfrak p_1$,
 where $\mathfrak k_1$ is the $(+1)$-eigenspace 
 and $\mathfrak p_1$ is the $(-1)$-eigenspace as follows:
\[
 \mathfrak{k}_1 =\left\{
\begin{pmatrix}
 a_{11} & a_{12} & a_{13} & 0\\
 a_{21} & a_{22} & 0 &  -a_{13}\\
 a_{31} & 0 & -a_{22} & -a_{12}\\
 0 & -a_{31} & -a_{21} & -a_{11}
\end{pmatrix}
\in\mathfrak{g}
\right\}, \quad 
 \mathfrak p_1 =
\left\{
\begin{pmatrix}
 a_{11} & a_{12} & a_{13} & a_{14}\\
 a_{21} & -a_{11} & a_{23} &  a_{13}\\
 a_{31} & a_{32} & -a_{11} & a_{12}\\
 a_{41} & a_{31} & a_{21} & a_{11}
\end{pmatrix}
\in\mathfrak{g}
\right\}.
\]
We decompose the Maurer-Cartan form according to this decomposition 
 $\alpha = F^{-1} d F = U dz + Vd\bar{z}$ along the 
 Lie algebra decomposition 
 $\mathfrak g = \mathfrak k_1 \oplus \mathfrak p_1$.
 First we decompose $U$ and $V$ as
 \[
 U = U_{\mathfrak k_1} + U_{\mathfrak p_1}, 
 \quad 
 V = V_{\mathfrak k_1} + V_{\mathfrak p_1},
 \quad U_{\mathfrak k_1},V_{\mathfrak k_1}\in\mathfrak{k}_1,
 \quad U_{\mathfrak p_1},V_{\mathfrak p_1}\in \mathfrak{p}_1.
 \]
  Next, set $\alpha_{\mathfrak k_1}=U_{\mathfrak k_1}dz+V_{\mathfrak k_1}d\bar{z}$ and 
  $\alpha_{\mathfrak p_1}=U_{\mathfrak p_1}dz+V_{\mathfrak p_1}d\bar{z}$, then we obtain the expression 

 \[
 \alpha = \alpha_{\mathfrak k_1} + \alpha_{\mathfrak p_1} =
 U_{\mathfrak k_1}dz + V_{\mathfrak k_1}d\bar{z} 
 + U_{\mathfrak p_1}dz + 
 V_{\mathfrak p_1}d\bar{z},
 \]
 where $U = U_{\mathfrak k_1} + U_{\mathfrak p_1}$ and
 $V = V_{\mathfrak k_1} + V_{\mathfrak p_1}$.
 Let us insert the spectral parameter $\lambda \in \mathbb{S}^1$
 into $U$ and $V$ as follows:
 \[ 
 U^{\lambda} = U_{\mathfrak k_1} + \lambda^{-1} U_{\mathfrak p_1} \;\;\mbox{and}  \;\;
 V^{\lambda} = V_{\mathfrak k_1} + \lambda V_{\mathfrak p_1}.
 \]
 Then a $\mathbb{S}^1$-family of $1$-forms $\alpha_{\lambda}$ 
 is defined as follows:
\begin{equation}\label{eq:alphalambda2} 
 \alpha^{\lambda} = 
 \alpha_{\mathfrak k_1} + \lambda^{-1} \alpha_{\mathfrak p_1}^{\prime}
 + \lambda \alpha_{\mathfrak p_1}^{\prime \prime} 
 = U^{\lambda} dz + V^{\lambda} d\bar{z}.
\end{equation}
 Using the matrices $U^{\lambda}$ 
 and $V^{\lambda}$, they are explicitly given 
 as follows:
\begin{equation*}\label{eq:extendedUV2}
 U^{\lambda}= 
 \begin{pmatrix}
 \bar{b}_z/(2\bar{b}) & \lambda^{-1} P & \lambda^{-1} k  &  \lambda^{-1}b \bar{P}  \\[0.1cm]
 \lambda^{-1}  & - \bar{b}_z/(2\bar{b}) & 0 & \lambda^{-1} k \\[0.1cm]
 0&\lambda^{-1} b& \bar{b}_z/(2\bar{b}) & \lambda^{-1} P \\[0.1cm]
 0 & 0 & \lambda^{-1} &- \bar{b}_z/(2\bar{b})
 \end{pmatrix}, \quad 
 V^{\lambda}= 
 \begin{pmatrix}
 b_{\bar z}/(2 b) & \lambda \bar{k} &  \lambda\bar{P}  & \lambda \bar{b}\bar{P} \\[0.1cm]
 0  & b_{\bar z}/(2 b) & \lambda \bar b & \lambda \bar P \\[0.1cm]
 \lambda &0& -b_{\bar z}/(2 b) & \lambda \bar{k} \\[0.1cm]
 0 & \lambda  & 0 &-b_{\bar z}/(2 b)
 \end{pmatrix}.
\end{equation*}
After these preparation,  we obtain 
the following theorem.
\begin{Theorem}\label{thm:Demoulin}
 Let $\mathfrak f$ be a surface in $\mathbb{R}\mathbb{P}^3$ with 
 positive definite projective metric 
 and $\gf$ the first-order
 Gauss map defined  in \eqref{eq:Gaussmap1}.
 Moreover, let $\{\alpha^{\lambda}\}_{\lambda
 \in \mathbb{S}^{1}}$  be a family of $1$-forms defined in 
 \eqref{eq:alphalambda2}. Then 
 the following three properties are mutually equivalent$:$
\begin{enumerate}
\item[$1.$] The surface $\mathfrak f$ is a definite Demoulin surface 
 or a projective minimal 
 coincidence surface. 
\item[$2.$]  The first-order Gauss map $\gf$ 
 is a harmonic map into $\mathcal{Q}$.
\item[$3.$]  $\{d + \alpha^{\lambda}\}_{\lambda\in\mathbb{S}^1}$ is a family of flat connections 
 on $\mathbb{D} \times G$.
\end{enumerate}
\end{Theorem}
\begin{proof}
 Let us first compute the flatness 
 \[
 d \alpha^{\lambda} + 
 \frac{1}{2}[\alpha^{\lambda} \wedge \alpha^{\lambda}]=0,
 \;\;\ \lambda\in\mathbb{S}^1
 \]
 for the connection $d + \alpha^{\lambda}$ on 
 $\mathbb{D}\times G$. 
 A straightforward computation shows that 
 $d \alpha^{\lambda} + 
 \tfrac{1}{2}[\alpha^{\lambda} \wedge \alpha^{\lambda}]=0$ holds for all $\lambda\in\mathbb{S}^1$ if and only if 
 \begin{equation}\label{eq:ZCR1st}
 P_{\bar z} = 0,
  \;\; \; 
  k_{\bar z} +k   \frac{b_{\bar z}}{b} =0,  \;\; \; 
 {\bar b} P_z + 2 {\bar b}_z P =0.
 \end{equation}
 On can see that this system implies the 
 projective Gauss-Codazzi equations 
 \eqref{eq:comp1}--\eqref{eq:comp2}. 
 In particular, the third equation is nothing but 
 the projective minimality equation \eqref{eq:projmin}.
 
 Every definite Demoulin surface clearly satisfies the above flatness condition (zero curvature equations) since 
 $P=0$. 
 
Assume that $P\neq 0$. The first equation of \eqref{eq:ZCR1st} means that 
$Pdz^2$ is a holomorphic differential. 
From the third equation together with the holomorphicity of $P$, one can deduce that $(\log \bar{b})_z$ is holomorphic.
Hence $(\log b/\bar{b})_{z\bar{z}} =0$. 
Via the holomorphic coordinate change of $z$ preserving the form of canonical system,
 we can assume that 
 $b = \bar{b}$, \textit{i.e.,} $b$ is real
  \footnote{The transformation rule  
 of $b$ under the conformal change of coordinates $w(z)$ 
  is given by $\tilde b = (\bar w_{\bar z}/w_z^2) b$, and thus 
   $\tilde b = \bar {\tilde b}$ can be achieved by a suitable choice of the function $w(z)$ 
   under the condition $(\log  b/\bar b)_{z \bar z}=0$, see~\cite[Section 3]{Ferapontov}.}. 
   Then \eqref{eq:kell} implies that $2k=b^2$. By using the second equation of \eqref{eq:ZCR1st}, 
 $b$ is constant and $k$ is a real constant.
 By using the third equation again, we get $P$ is constant.
 This implies that $P=p$ is a non-zero constant. 
 After these reparametrization, the canonical system 
 is rewritten as 
 \[
 f_{zz} = bf_{\bar z} + p f,\;\;
  f_{\bar{z}\bar{z}} = bf_{z} + p f.
 \]
 A surface satisfying the above equation 
 is a special case of the \emph{coincidence surface},~\cite[Example 2.19]{Sasaki}. In fact, it is easy to see 
 that the surface is a projective minimal coincidence surface.
 Thus the equivalence of the claim $1$ and claim $3$ follows.
 
 The equivalence of the claims $2$ and $3$ follows from 
Proposition \ref{prop:ZCR}, since the $\mathbb{S}^1$-family of 
 $1$-forms $\alpha^{\lambda}$ is given 
 by the involution $\tau_1$ and it defines 
 the semi-Riemannian symmetric space $\mathcal{Q}
 =\mathrm{SL}_{4}\mathbb{R}/K_1$.
\end{proof}

\begin{Corollary}\label{coro:Demoulin}
 Retaining the assumptions in Theorem $\ref{thm:Demoulin}$, 
 the following are equivalent$:$
\begin{enumerate}
\item[$1.$]  The surface $\mathfrak f$ is a definite Demoulin surface.
\item[$2.$]  The first-order Gauss map $\gf$ 
 is a conformal harmonic map into $\mathcal{Q}$.
\end{enumerate}
\end{Corollary}
\begin{proof}
 From Proposition \ref{prop:conformality}, it is 
 easy to see that the first-order Gauss map 
 is conformal if and only if it satisfies that $P=0$, that is, 
 the surface is a definite Demoulin surface. Moreover, from Theorem 
 \ref{thm:Demoulin} the Gauss map of the Demoulin surface 
 is harmonic.
\end{proof}

This corollary implies that if $\mathfrak f$ is a definite Demoulin 
 surface or a projective minimal coincidence surface, then there exists a $\mathbb{S}^1$-parameter family of smooth map
 $F_{\lambda}:\mathbb{D}\times\mathbb{S}^1
 \to G$ which is a solution to 
 \[
(F_{\lambda})^{-1}dF_{\lambda}=\alpha^\lambda
 \]
 under initial condition $F_{\lambda}(z_*)=\mathrm{id}$.
 One can see that $F_{\lambda}$ is regarded as 
 a smooth map of $\mathbb{D}$ into the following 
twisted loop group
\[
 \Lambda G_{\tau_1}= \{ g : \mathbb{S}^{1} \to G\;|\; 
 \tau_1 g(\lambda) = g (-\lambda)\}
\]
of $G$. 
The $\Lambda G_{\tau_1}$-valued map 
$F_{\lambda}$ is referred as to the 
\emph{extended Wilczynski frame} of a definite Demoulin surface.

Precisely speaking, the extended 
Wilczynski frame $F_{\lambda}$
is not the Wilczynski frame of a Demoulin surface or a projective minimal coincidence surface except for $\lambda =1$.
By conjugating $F_{\lambda}$ by $DF_{\lambda}D^{-1}$ with 
$D=\di (1, \lambda, \lambda^{-1}, 1)$, 
 the frames $DF_{\lambda}D^{-1}$ give 
 a family of Wilczynski frames for Demoulin surfaces
  or projective minimal coincidence surfaces.
 The corresponding Demoulin surfaces or projective 
 minimal coincidence surfaces have the same projective metric 
 $8 |b|^2 \>dz d\bar{z}$ but the different conformal classes of cubic forms
 $\lambda^{-3} b \>dz^3$. Moreover, 
 the functions $P$ changes as $\lambda^{-2} P$.

\section{Primitive lifts}\label{sec:Primitive}

 We now show that the extended Wilczynski frame for a Demoulin surface 
 has an additional 
 order three cyclic symmetry.
 Let $\sigma$ be an order three automorphism on the 
 complexification $\mathrm{SL}_{4}\mathbb{C}$ 
 of $G$ as follows:
\[
 \sigma X = \mathrm{Ad}(E) X, \;\; X \in \mathrm{SL}_{4}\mathbb{C},
\]
 where $E =\di (1, \epsilon^2, \epsilon, 1)$ with 
 $\epsilon = e^{2\pi \sqrt{-1}/3 }$. 
 It should be emphasized that $\sigma$ preserves the real form $G$. Thus $\sigma$ is regarded as an automorphisms of  $G$.
 
Next, one can check that 
 $F_{\lambda}$ satisfies the symmetry
 $\sigma (F_{\lambda}) =F_{\epsilon \lambda}$,
 since $U^{\lambda}$ 
 and $V^{\lambda}$ satisfy
 the same symmetry. It is also easy to see that 
 $\tau_1$ and $\sigma$ commute, and $\kappa = \tau_1 \circ \sigma$
 defines an automorphism of order six. We obtain a 
 regular semi-Riemannian $6$-symmetric space $G/K$ 
 (see  Appendix A.1), where 
 \begin{equation}\label{eq:stabilizer6}
 K =\{\di (k_1, k_2, k_2^{-1}, k_1^{-1}) \;|\;k_1
 \in \mathbb{R}^{\times}, k_2 \in \mathbb{S}^{1}\}
 \cong\mathrm{SO}_{1,1}\times\mathrm{SO}_2.
 \end{equation}%
Note that $G/K$ is identified with 
 $\{gJg^{\mathrm T}\>|\>g\in G\}$, where $J=EJ_1$.
 There is a homogeneous projection 
 \[
 \pi:G/K\to G/K_{1};\;\; gK\longmapsto gK_{1}.
 \]%
 The extended Wilczynski 
 frame $F_{\lambda}$ satisfies the symmetry
\[
 \kappa (F_{\lambda}) = F_{- \epsilon \lambda}.
\]
 Note that $-\epsilon$ is the $6$th root of unity. 
 From the above argument, it is easy to see 
 that the extended Wilczynski frame 
 $F_{\lambda}= F(\lambda)$ for a Demoulin surface 
 is an element of the twisted loop group of $G$:
 \[
 \Lambda G_{\kappa} = \{ g : \mathbb{S}^{1} \to G\;|\; 
 \kappa g(\lambda) = g (-\epsilon \lambda)\}.
 \]
\begin{Theorem}\label{coro:primitivity}
 The first-order Gauss map of a Demoulin surface, which 
 is conformal harmonic into $\mathcal{Q}= G/K_{1}$, 
 can be obtained by the homogeneous projection 
 of a primitive map into the regular semi-Riemannian $6$-symmetric space 
 $G/K \cong\mathrm{SL}_{4}\mathbb{R}/\mathrm{SO}_{1,1}\times\mathrm{SO}_2$.
\end{Theorem}
\begin{proof}
 The $0$th-eigenspace $\mathfrak{g}^{\mathbb C}_{0}$ 
 and $\pm1$st-eigenspaces 
 $\mathfrak{g}^{\mathbb C}_{\pm 1}$  of 
 the derivative of the order six automorphism 
 $\kappa =\tau_1 \circ \sigma$
 are described as follows:
\begin{equation*}
 \mathfrak g^{\mathbb C}_{0} = \left\{\di(a_{11}, a_{22}, - a_{22}, -a_{11})
 \;|\;a_{11}\in \mathbb{R},\;\;a_{22} \in \mathbb{C} \right\},
\end{equation*}
and 
\begin{equation*}
 \mathfrak g^{\mathbb C}_{-1} = \left\{
\left.
 \begin{pmatrix} 
0 & 0 & a_{13} & 0  \\ 
a_{21} & 0 & 0 & a_{13}  \\ 
0 & a_{32} & 0& 0  \\ 
0 & 0 & a_{21} & 0  
\end{pmatrix}
\right| a_{i j} \in \mathbb{C}
\right\},\quad
 \mathfrak g^{\mathbb C}_{1} =
\left\{
\left.
 \begin{pmatrix} 
0 & a_{12} & 0 & 0  \\ 
0 & 0 & a_{23} & 0  \\ 
a_{31} & 0 & 0& a_{12}  \\ 
0 & a_{31} & 0 & 0  
\end{pmatrix}
\right| a_{i j} \in \mathbb{C}
\right\}.
\end{equation*}
 From the matrices $U^{\lambda}$ and $V^{\lambda}$
 in \eqref{eq:alphalambda2}  with $P=0$, we see that 
 the condition in Definition \ref{def:primitivemap}
 of primitive map is satisfied.
 The stabilizer of $\kappa$ is the closed subgroup $K$ given by 
 \eqref{eq:stabilizer6}.
 Therefore there is a primitive map $g= F J F^{\mathrm T}$
 $J=E J_1$ into the $6$-symmetric space $G/K$ such that 
$\pi\circ g=\mathrm{Ad}(L^{-1})g_1$.
 Since 
 $\mathrm{Ad}(L^{-1}):\mathrm{SL}_{4}\mathbb{R}/\hat{K}_1\to G/K_1$ is 
 an isometry, $g_1=\mathrm{Ad}(L)(\pi\circ g)$ is harmonic.
\end{proof}

This theorem enable us to establish a 
generalized Weierstrass type representation for 
definite Demoulin surfaces by virtue of 
primitive maps into the 
semi-Riemannian $6$-symmetric space $G/K$, see~\cite{DMPW}.

\begin{Remark}
{\rm 
Corresponding result theorem for 
indefinite Demoulin surfaces was obtained by the second named author in the preprint 
version of~\cite{Kobayashi2015}. 
}
\end{Remark}

\appendix
\section{Primitive harmonic maps}
\subsection{Homogeneous geometry}\label{sc:homogeneous}
Let $G$ be a semi-simple real Lie group with automorphism 
$\tau$ of order $k\geq 2$. We consider a 
reductive homogenous space $G/K$ equipped with a 
$G$-invariant semi-Riemannian metric satisfying the 
following three conditions:
\begin{itemize}
\item The closed subgroup $H$ satisfies
$G^{\circ}_{\tau}\subset K\subset G_{\tau}$. 
Here $G_{\tau}$ is the Lie subgroup of all fixed points of 
$\tau$ and $G^{\circ}_{\tau}$ the identity component of it.
\item The $G$-invariant semi-Riemannian metric is derived from (a constant multiple of) the Killing form of
$G$. 
\item The Lie algebra $\mathfrak{k}$ of $K$ is \emph{non-degenerate} 
with respect to the induced scalar product.
\end{itemize} 
The resulting homogeneous semi-Riemannian  space 
$G/K$ is called a regular \emph{semi-Riemannian 
$k$-symmetric space}. 
Note that a regular semi-Riemannian $2$-symmetric spaces is just a 
semi-Riemannian symmetric space. 
Since $\mathfrak{k}$ is non-degenerate, 
the orthogonal complement $\mathfrak{p}$ of 
$\mathfrak{k}$ is non-degenerate and 
can be identified with the tangent space of $G/K$ at the origin $o=K$. 
The Lie algebra $\mathfrak{g}$ is decomposed into the direct sum:
\[
\mathfrak{g}=\mathfrak{k}\oplus \mathfrak{p}.
\]
of linear subspaces. 

We denote the induced Lie algebra 
automorphism of $\mathfrak{g}$ by the same letter $\tau$. 
Now we have the eigenspace decomposition of the 
complexified Lie algebra $\mathfrak{g}^{\mathbb C}$;
\[
\mathfrak{g}^{\mathbb C}=\sum_{j\in\mathbb{Z}_k}
\mathfrak{g}_{j}^{\mathbb C},
\]
where $\mathfrak{g}_{j}^{\mathbb C}$ is the eigenspace of 
$\tau$ with eigenvalue $\omega^j$. 
Here $\omega$ is the (primitive) $k$-th root of unity. 
In particular, $\mathfrak{g}^{\mathbb C}_0 =\mathfrak{k}^{\mathbb C}$ 
and $\mathfrak{g}^{\mathbb C}_{-1} = \overline{\mathfrak{g}^{\mathbb C}_1}$.
 Let us define a subbundle $[\mathfrak{g}^{\mathbb C}_j]$ 
 of $G/K\times \mathfrak{g}$ by 
 \[
 [\mathfrak{g}^{\mathbb C}_j]_{g\cdot o}=\mathrm{Ad}(g)
 \mathfrak{g}^{\mathbb C}_j.
 \]
Then the complexified tangent bundle $T^{\mathbb C}G/K$ is 
expressed as 
\[
T^{\mathbb C}\,G/K=\sum_{j\in\mathbb{Z}_k,\,j\not=0}[\mathfrak{g}
^{\mathbb{C}}_j].
\]

\subsection{Primitive maps}\label{sc:primitive}

A smooth map $\psi:\varSigma \to N$ of a Riemann surface $\varSigma$ 
into a semi-Riemannian manifold $N$ 
is said to be a \emph{harmonic map} if its 
\emph{tension field} $\mathrm{tr}(\nabla d\psi)$ vanishes. 

For smooth maps into regular semi-Riemannian $k$-symmetric spaces with $k>2$, 
the notion of primitive map was introduced by Burstall-Pedit~\cite{BP94} (see also 
Bolton-Pedit-Woodward~\cite{BPW}).

\begin{Definition}\label{def:primitivemap}
{\rm Let $\psi:\varSigma\to G/K$ be a smooth map of a Riemann surface 
$\varSigma$ into a regular 
 semi-Riemannian $k$-symmetric space with $k>2$. 
 Then $\psi$ is said to be a \emph{primitive map} if 
 $\mathrm{d}\psi(T^{\prime}\varSigma)\subset 
 [\mathfrak{g}_{_{-1}}^{\mathbb C}]$. Here $T^{\prime}\varSigma$ denotes the 
 $(1,0)$-tangent bundle of $\varSigma$.
 }
 \end{Definition}%
 
Black~\cite{Black} showed that primitive maps are 
\emph{equi-harmonic}, that is, 
harmonic with respect to suitable invariant metrics on $G/K$
(see also~\cite{BP94}). In addition primitive maps well behave
with respect to homogeneous projections~\cite[Theorem 3.7]{BP94}. 
\begin{Theorem}
\label{thm:BP94}
Let $H$ be a closed subgroup of $G$ satisfying 
\begin{itemize}
\item $K\subset H$.
    \item The Lie algebra $\mathfrak{h}$ of $H$ is non-degenerate.
    \item The decomposition $\mathfrak{g}=\mathfrak{h}\oplus\mathfrak{q}$ is reductive 
    and stable under $\tau$. Here $\mathfrak{q}$ is the orthogonal complement of $\mathfrak{h}$.
\end{itemize}
 Denote by $\pi_H:G/K\to G/H$ be the homogenous projection. Then for any primitive map $\psi$, $\pi_{H}\circ \psi$ is a
 harmonic map into $G/H$.
 \end{Theorem}
 Note that when $k=2$, $[\mathfrak{g}^{\mathbb C}_{_{-1}}]=T^{\mathbb C}\,G/K$ and 
 the primitivity condition is vacuous. On the other hand when $k>2$, every
 primitive map is harmonic with respect to the Killing metric. 
 To provide a unified description, we recall the following terminology 
from~\cite{BP95}.
\begin{Definition}
{\rm 
A smooth map $\psi:\varSigma\to G/K$ into a regular 
 semi-Riemannian $k$-symmetric space is said to be a 
\emph{primitive harmonic map} if it is primitive for $k>2$ and harmonic if $k=2$. 
}
\end{Definition}

 Now let $\psi:\mathbb{D} \to G/H$ be a smooth map from a simply connected
 Riemann surface $\mathbb{D}$ into a regular semi-Riemannian $k$-symmetric space $G/K$ with $k\geq 2$. Take a frame $\Psi:\mathbb{D} \to G$ of
 $\psi$ and put $\alpha:=\Psi^{-1}d\Psi$. Then we have the
 identity (\emph{Maurer-Cartan equation}):
\[
d\alpha+\frac{1}{2}[\alpha \wedge \alpha]=0.
\]
 Decompose $\alpha$ along the Lie algebra decomposition
 $\mathfrak{g}=\mathfrak{k}\oplus\mathfrak{p}$ as
\[
\alpha=\alpha_{\mathfrak{k}}+\alpha_{\mathfrak{p}},
\quad 
\alpha_{\mathfrak{k}}\in
\mathfrak{k},
\quad 
\alpha_{\mathfrak{p}}
\in
\mathfrak{p}.
\]
We decompose $\alpha_{\mathfrak{p}}$ 
with respect to the conformal
structure of $\mathbb{D}$ as
\[
\alpha_{\mathfrak p}=
\alpha_{\mathfrak p}^{\prime}+
\alpha_{\mathfrak p}^{\prime \prime}.
\]
 Here $\alpha_{\mathfrak{p}}^{\prime}$ and $\alpha_{\mathfrak p}^{\prime
 \prime}$ are the $(1,0)$ and $(0,1)$-part of $\alpha_{\mathfrak p}$,
 respectively. Since $G$ is a real Lie group, $\alpha^{\prime
 \prime}_{\mathfrak p}$ is the conjugate of $\alpha^{\prime}_{\mathfrak
 p}$. 
 
 Now let us assume that $\psi$ is a 
 primitive harmonic map, then 
 $\alpha_{\mathfrak{p}}^{\prime}$ is $[\mathfrak{g}_{_{-1}}^{\mathbb C}]$-valued
 and $\alpha_{\mathfrak{p}}^{\prime\prime}$ is 
 $[\mathfrak{g}_{_1}^{\mathbb C}]$-valued, respectively.
 Hence the decomposition of $\alpha$ is rewritten as

 \[
 \alpha=\alpha_{-1}^{\prime}+\alpha_{0}+\alpha_{1}^{\prime}.
 \]
Now let us introduce a spectral parameter 
$\lambda\in\mathbb{S}^1$ into $\alpha$ as 
\[
\alpha^{\lambda}:=
\alpha_{0}
+\lambda^{-1}\alpha_{-1}^{\prime}+
\lambda\> \alpha_{1}^{\prime \prime}.
\] 
We arrive at the \emph{zero curvature representation} of primitive harmonic maps:
\begin{Proposition}\label{prop:ZCR}
  Let $\mathbb D$ be a connected open subset of $\mathbb{C}$. Let
 $\psi:\mathbb{D} \to G/K$ be a primitive harmonic map.
 Then the loop of connections $d+\alpha^{\lambda}$ is flat for
 all $\lambda$, that is,
\[
d\alpha^{\lambda}+\frac{1}{2}
[\alpha^{\lambda}\wedge
\alpha^{\lambda}]=0
\]
 for all $\lambda$.

 Conversely assume that $\mathbb{D}$ is simply connected. Let
 $\alpha^{\lambda}=\alpha_{0}+\lambda^{-1} \alpha_{-1}^{\prime}+\lambda\alpha_{1}^{\prime\prime}$ be an
 $\mathbb{S}^1$-family of $\mathfrak{g}$-valued one-forms which satisfies 
\[
d\alpha^{\lambda}+\frac{1}{2}
[\alpha^{\lambda}\wedge
\alpha^{\lambda}]=0
\]
for all $\lambda \in \mathbb{S}^1$.
Then there exists a one-parameter family of maps
$\Psi_{\lambda} :\mathbb{D} \to G$ such that
\[
\Psi_{\lambda}^{-1}
d\Psi_{\lambda}=\alpha^{\lambda},
\]
and 
\[
\psi_{\lambda}=\Psi_{\lambda}\;{\mbox{\rm mod}}\; K:\mathbb{D} \to G/K
\]
is primitive harmonic for all $\lambda$.
\end{Proposition}

\section{Projective minimal surfaces and the conformal Gauss maps}
 \label{sc:confGauss}

\subsection{Conformal Gauss map}
Let $\mathfrak f: M \to \mathbb{R}\mathbb{P}^3$ be a surface with
 Wilczynski frame $F$ as in subsection \ref{subsc:Gaussmap}. 
 We define a map $\gs$ by
 \[
 \gs = \hat{F} \hat{J}_{2} \hat{F}^{\mathrm T}
 = -\mathrm{Ad}(L)(F J_2 F^{\mathrm T}),
 \quad 
 \hat{J}_{2}=-LJ_{2}L^{\mathrm T},
 \]
 where $\js = \offdi (1, -1, -1, 1)$ 
 (\textit{cf}.~\cite[\S 4.1]{SasakiRokko}).
 Analogous to the first-order Gauss map $\gf$, 
 $\gs$ takes value in the space $\mathcal{Q}$ of conformal 
 $2$-spheres in $\mathbb{R}\mathbb{P}^3$. More precisely,
since the matrix $\hat{J}_2$ is of signature $(1,3)$, it is a point of $\mathcal{Q}$. Thus $\mathcal{Q}$ is realized as a
homogeneous space $\mathrm{SL}_{4}\mathbb{R}/\hat{K}_2$, where 
$\hat{K}_2$ is the stabilizer at $\hat{\js} \in \mathcal{Q}$ explicitly given by
 $\hat{K}_2 = \{ X \in \SL \;|\; X \js X^{\mathrm T} = \js\}$, which 
 is also isomorphic to ${\rm SO}^{+}_{1, 3}$. 
 Thus the map $\gs$ takes value in $\mathrm{SL}_{4}\mathbb{R}/\hat{K}_2$:
 \begin{equation}\label{eq:Gaussmap2}
 \gs : M \to \mathcal{Q} \cong \mathrm{SL}_{4}\mathbb{R}/\hat{K}_2= \mathrm{SL}_{4}\mathbb{R}/\mathrm{SO}^{+}_{1,3}.
 \end{equation}
 This map $\gs$ is known to be a \emph{Lie quadric} which has 
 the second order contact to the surface, 
 see~\cite[Section 18]{Lane}.
 The map $\gs$ has been called the 
 \emph{conformal Gauss map} for a surface $\mathfrak f$
 in $\mathbb{R}\mathbb{P}^3$, see~\cite{Th, BHJ}. 
 In~\cite{MN}, the conformal Gauss map $\gs$ was called 
 the \emph{projective Gauss map}. 
 In classical literature, $\gs$ was called the 
 \emph{congruence of Lie quadrics}.

\begin{Proposition}[Theorem 3 in~\cite{BHJ}]\label{prop:conformality2}
 The conformal Gauss map $\gs$ is conformal map.
\end{Proposition}
\begin{proof}
 As in the proof of Proposition \ref{prop:conformality},
 a direct computation shows that
\[
 \partial_{z}\gs = -2 (LF) \di (b \bar{P}, 0, -b, 0) (LF)^{\mathrm T}
 \;\;\mbox{and}\;\;
 \partial_{\bar z}{\gs}= -2 (LF) \di (\bar{b} P, -\bar{b}, 0, 0) (LF)^{\mathrm T}.
\]
 Thus 
\[
\langle \partial_{z}\gs, \partial_{z}\gs \rangle = \langle  \partial_{\bar z}{\gs},  \partial_{\bar z}{\gs}\rangle  = 0
\;\; \mbox{and} \;\;
\langle  \partial_{z}\gs,  \partial_{\bar z} \gs\rangle = 
4 |b|^2  \neq  0.
\]
 Since the coordinates $(z, \bar{z})$ are null for the conformal 
 structure induced by $\mathfrak f$, 
 the conformal Gauss map $\gs$ is conformal.
\end{proof}

\begin{Remark}
The \emph{Hodge star operator} $\star$ on $\wedge^{2}
\mathbb{R}^{1,3}$ is introduced by 
\[
\langle a,b\rangle=\Omega(a\wedge \star b).
\]
Since $\langle\cdot,\cdot\rangle$ is Lorentzian, 
$\star$ satisfies $\star^2=-1$. Thus 
the complexification $(\wedge^{2}\mathbb{R}^{1,3})
^{\mathbb C}\cong \wedge^{2}
\mathbb{C}^{1,3}$ has the eigenspace decomposition
\[
(\wedge^{2}\mathbb{R}^4)
^{\mathbb C}=\mathrm{S}\oplus\overline{\mathrm{S}},
\]
where $\mathrm{S}$ is the $\sqrt{-1}$-eigenspace of $\star$.
In this way, a quadric $Q\in\mathcal{Q}$ corresponds to a complex linear subspace 
$\mathrm{S}$ of $(\wedge^{2}\mathbb{R}^4)
^{\mathbb C}$. The correspondence $Q\longmapsto
\mathrm{S}$ defines a 
smooth bijection from the space $\mathcal{Q}$ of conformal $2$-spheres in $\mathbb{R}\mathbb{P}^3$ 
to the space 
\[
\mathcal{G}_{2,0}^{3,3}=
\{\mathrm{S}\subset (\mathbb{R}^{3,3})^{\mathbb C}\>|\>
\mathrm{S}\cap \mathrm{S}^{\perp}=\{0\},\ \ 
\bar{\mathrm{S}}=\mathrm{S}^{\perp}\}.
\]
Under this identification, $g_2$ is regarded as 
a smooth map into $\mathcal{G}_{2,0}^{3,3}$ in~\cite[p.~183]{BHJ}, \cite[p.~30]{Clarke}.
\end{Remark}

\subsection{Projective minimal surfaces and the conformal Gauss maps}
The space $\mathcal{Q}$ of conformal $2$-spheres in 
$\mathbb{R}\mathbb{P}^3$ is isomorphic to 
the semi-Riemannian symmetric space $G/K_2$, where 
\[
K_2=\{a\in G\>|\>aJ_2a^{\mathrm T}=J_2\}.
\]
 Let $\tau_2$ be the outer involution on 
 $G$ associated to the symmetric 
 space $G/K_2$ defined by: 
\[
 \tau_2 (a) = \js \left(a^{\mathrm T}\right)^{-1}\js ,
\] 
 where $a \in G$. By abuse of notation, we denote the 
 differential of $\tau_2$ by the same letter $\tau_2$  
 which is an outer involution on $\mathfrak{g}$:
\begin{equation}\label{eq:tau1}
\tau_2 (X) = - \js X^{\mathrm T} \js,
\end{equation}
 where $X \in \mathfrak{g}$.
 Let us consider the eigenspace decomposition
 of $\mathfrak g$ with respect to $\tau_2$, that 
 is, $\mathfrak g = \mathfrak k_2 \oplus \mathfrak p_2$,
 where $\mathfrak k_2$ is the $(+1)$-eigenspace 
 and $\mathfrak p_2$ is the $(-1)$-eigenspace as follows:
\[
 \mathfrak k_2 =\left\{
\begin{pmatrix}
 a_{11} & a_{12} & a_{13} & 0\\
 a_{21} & a_{22} & 0 &  a_{13}\\
 a_{31} & 0 & -a_{22} & a_{12}\\
 0 & a_{31} & a_{21} & -a_{11}
\end{pmatrix}
\in\mathfrak{g}
\right\}, \quad
 \mathfrak p_2 =
\left\{
\begin{pmatrix}
 a_{11} & a_{12} & a_{13} & a_{14}\\
 a_{21} & -a_{11} & a_{23} &  -a_{13}\\
 a_{31} & a_{32} & -a_{11} & -a_{12}\\
 a_{41} & -a_{31} & -a_{21} & a_{11}
\end{pmatrix}
\in\mathfrak{g}
\right\}.
\]
 According to this decomposition  
 $\mathfrak g = \mathfrak k_2 \oplus \mathfrak p_2$, 
 the Maurer-Cartan form $\alpha = F^{-1} d F = U dz + Vd\bar{z}$
 can be decomposed into  
\[
 \alpha = \alpha_{\mathfrak k_2} + \alpha_{\mathfrak p_2} =
 U_{\mathfrak k_2}dz + V_{\mathfrak k_2}d\bar{z} + U_{\mathfrak p_2}dz + 
 V_{\mathfrak p_2}d\bar{z},
\]
 where $U = U_{\mathfrak k_2} + U_{\mathfrak p_2}$ and
 $V = V_{\mathfrak k_2} + V_{\mathfrak p_2}$.
 Let us insert the parameter $\lambda \in \mathbb S^1$ 
 into $U$ and $V$ in a manner similar to section 
 \ref{subsection6.2}:
\[
 U^{\lambda}= U_{\mathfrak k_2} + \lambda^{-1} U_{\mathfrak p_2} \;\;\mbox{and}  \;\;
 V^{\lambda} = V_{\mathfrak k_2} + \lambda V_{\mathfrak p_2}.
\]
 Then a family of $1$-forms $\alpha^{\lambda}$ 
 is defined as follows:
\begin{equation}\label{eq:alphalambda1} 
 \alpha^{\lambda} = 
 \alpha_{\mathfrak k_2} + \lambda^{-1} \alpha_{\mathfrak p_2}^{\prime}
 + \lambda \alpha_{\mathfrak p_2}^{\prime \prime} 
 = U^{\lambda} dz + V^{\lambda} d\bar{z}.
\end{equation}
 In fact the matrices $U^{\lambda}$ 
 and $V^{\lambda}$ are explicitly given 
 as follows:
\begin{equation}\label{eq:extendedUV1}
 U^{\lambda}= 
 \begin{pmatrix}
 \frac{\bar{b}_z}{2 \bar{b}} & P & k  &  \lambda^{-1}b \bar{P}\\[0.1cm]
 1  & -\frac{\bar{b}_z}{2 \bar{b}} & 0 & k \\[0.1cm]
 0&\lambda^{-1} b&\frac{\bar{b}_z}{2 \bar{b}} & P \\[0.1cm]
 0 & 0 & 1 &-\frac{\bar{b}_z}{2 \bar{b}}
 \end{pmatrix}, \quad
 V^{\lambda}= 
 \begin{pmatrix}
 \frac{b_z}{2 b} & \bar{b} &\bar{P}  & \lambda \bar{b} P \\[0.1cm]
 0  & \frac{b_z}{2 b} & \lambda \bar{b} & \bar{P} \\[0.1cm]
 1&0& -\frac{b_z}{2 b} & \bar{k} \\[0.1cm]
 0 & 1 & 0 &-\frac{b_z}{2 b}
 \end{pmatrix}.
\end{equation}
 Then the projective minimal surface can be characterized 
 by the harmonicity of the conformal Gauss map~\cite{Th}, \cite[Theorem 7]{BHJ},
 and by a family of flat connections.
\begin{Theorem}[\cite{Th}, \cite{BHJ}]\label{thm:projmini}
 Let $\mathfrak f$ be a surface in $\mathbb{R}\mathbb{P}^3$ and $\gs$ the 
 conformal Gauss map defined in \eqref{eq:Gaussmap2}.
 Moreover, let 
 $\{\alpha^{\lambda}\}_{\lambda
 \in \mathbb S^{1}}$  be a family of $1$-forms defined in 
 \eqref{eq:alphalambda1}. Then 
 the following are mutually equivalent$:$
\begin{enumerate}
\item[$1.$]  The surface $\mathfrak f$ is a projective minimal surface. 
\item[$2.$]  The conformal Gauss map $\gs$ 
 is a conformal harmonic map into $\mathcal{Q}$.
\item[$3.$]  $\{\alpha^{\lambda}\}_{\lambda
 \in \mathbb{S}^{1}}$ is a family of flat connections 
 on $\mathbb{D} \times G$.
\end{enumerate}
\end{Theorem}
\begin{proof}
 Let us compute the flatness conditions of $d + \alpha^{\lambda}$, 
 that is, the Maurer-Cartan equation $ d \alpha^{\lambda} + 
 \tfrac{1}{2}[\alpha^{\lambda} \wedge \alpha^{\lambda}] =0$. 
 It is easy to see that except for the $(1, 4)$-entry, the 
 Maurer-Cartan equation is equivalent to \eqref{eq:comp1}.
 Moreover, the $\lambda^{-1}$-term and the $\lambda$-term 
 of the $(1, 4)$-entry are equivalent to 
 that the first equation and the second equation 
 in \eqref{eq:projmin}, respectively. 
 Thus the equivalence of $(1)$ and $(3)$ follows.
 
 The equivalence of $(2)$ and $(3)$ follows from 
 Proposition \ref{prop:ZCR}, since the family of 
 $1$-forms $\alpha^{\lambda}$ is given 
 by the involution $\tau_2$ and it defines 
 the semi-Riemannian symmetric space $\mathcal{Q}=\mathrm{SL}_{4}\mathbb{R}/K_2$.
\end{proof}
\begin{Remark}
{\rm  The above theorem implies that if $\mathfrak f$ 
 is a projective 
 minimal surface, then there exists a family of projective 
 minimal surface $\mathfrak f^{\lambda}\> (\lambda \in \mathbb S^1)$
 such that $\mathfrak f^{\lambda}|_{\lambda =1}=\mathfrak f$. 
 Projective minimal surfaces of 
 the family have the same projective metric 
 $8|b|^2\,dz d\bar{z}$
 but the different conformal classes of 
 cubic forms $\lambda^{-1} b \>dz^3$. 
 }
\end{Remark}


\begin{thebibliography}{99}

\bibitem{Black}
M.~Black, 
\emph{Harmonic maps into Homogeneous Spaces}, 
Pitman Research Notes in Mathematics Series \textbf{255}, 1991.

\bibitem{BPW}
J.~Bolton, 
F.~Pedit, 
L.~M.~Woodward, 
Minimal surfaces and the affine Toda field model, 
\emph{J. Reine Angew. Math.} \textbf{459} (1995), 119--150.

\bibitem{BW1992}
J.~Bolton, 
L.~M.~Woodward, 
Congruence theorems for harmonic maps 
from a Riemann surface into $\mathbb{C}\mathrm{P}^n$ and 
$\mathrm{S}^n$, 
\emph{J. London Math. Soc.} (2) \textbf{45} (1992), 363--376.

\bibitem{BHJ}
F.~E.~Burstall, 
U.~Hertrich-Jeromin, 
Harmonic maps in unfashionable geometries, 
\emph{Manuscripta. Math.} \textbf{108} (2002), 
171--189.

\bibitem{BP94}
F.~E.~Burstall, 
F.~Pedit, 
Harmonic maps via Adler-Kostant-Symes theory, in:  
\emph{Harmonic Maps and Integrable
Systems} (A.~P.~Fordy, J.~C.~Wood eds.), 
Aspects of Mathematics \textbf{E23} (1994), 
Vieweg, pp.~221--272.



\bibitem{BP95} 
F.~E.~Burstall, 
F.~Pedit, 
Dressing orbits of harmonic maps, 
\emph{Duke Math. J.} \textbf{80} (1985), 353--382.

\bibitem{Clarke}
D.~J. ~Clarke, 
\emph{Integrability in Submanifold Geometry}, 
Thesis, Univ. Bath., 2012.

\bibitem{CW}
S.~S.~Chern, J.~G.~Wolfson, 
Harmonic maps of the two-sphere
into a complex Grassmannian manifolds II,
\emph{Ann. of Math.} (2) \textbf{125} (1987), 301--335.

\bibitem{Darboux}
G.~Darboux, 
\emph{Le\c{c}ons sur la 
th{\'e}orie g{\'e}n{\'e}rale des surfaces} 
I, 1914, II, 1915, second edition, Gauthier-Villars. 


\bibitem{Demoulin}
A.~Demoulin, 
Sur deux transformations des surfaces dont 
les quadriques de Lie n'ont que deux ou trois
points caracteristiques, 
\emph{Bull de l'Acad Belgique} \textbf{19} (1933), 
479 --502, 579--592, 1352--1363.

\bibitem{DFKW}
J.~F.~Dorfmeister, 
W.~Freyn, 
S.-P.~Kobayashi, E.~Wang, 
Survey on real forms of the complex 
$A_{2}^{(2)}$-Toda equation and surface theory.
\emph{Complex Manifolds} \textbf{6} (2019), 194--227.

\bibitem{DMPW}
J.~Dorfmeister, 
I.~McIntosh, 
F.~Pedit, 
H.~Wu, On the meromorphic potential for a harmonic surface
in a $k$-symmetric space, 
\emph{Manuscripta Math.} \textbf{92} (1997) no.~2, 143--152.

\bibitem{Ferapontov}
F.~E.~Ferapontov, 
Integrable systems in projective differential geometry, 
\emph{Kyushu J. Math.} \textbf{54} (2000), no.~1, 183--215.

\bibitem{Hu}
H.~S.~Hu, 
Darboux transformations of Su-chain, in: 
\emph{Differential Geometry. Proceedings of the Symposium in 
honor of Professor Su Buchin on his 90th Birthday} 
(C.~H.~Gu, H.~S.~Hu, Y.~ L.~Xin eds.), 
World Scientific, 1993, pp.~325--380.

\bibitem{Kobayashi2015} 
 S.-P.~Kobayashi, 
 A loop group method for Demoulin surfaces in the 3-dimensional real projective space, 
 \emph{Differential Geom. Appl.}, \textbf{40} (2015), 57--66 
(\texttt{arXiv:1301.6325v2[math.DG]})%

\bibitem{Lane}
 E.~P.~Lane, 
 \emph{Projective Differential Geometry of Curves and Surfaces}, 
 University of Chicago Press, 1932.
 
\bibitem{LM}
C.~LeBrun, 
L.~J.~Mason, 
Nonlinear gravitons, null geodesics, 
and holomorphic disks, 
\emph{Duke Math. J.} \textbf{136} (2007), 205--273.

\bibitem{MN}
E.~Musso, 
L.~Nicolodi, 
Tableaux over Lie algebras, integrable systems, and 
classical surface theory, \emph{Comm. Anal. Geom.} 
\textbf{14} (2006), 475--496.

\bibitem{NS}
K.~Nomizu, T.~Sasaki, 
Centroaffine immersions of codimension two and projective 
hypersurface theory, 
\emph{Nagoya Math. J.} \textbf{132} (1993), 63--90.

\bibitem{NSBook}
K.~Nomizu, T.~Sasaki, 
\emph{Affine Differential 
Geometry. Geometry of 
Affine Immersions}, 
Cambridge Tracts in Mathematics \textbf{11} (1994), 
Cambridge University Press.

\bibitem{Pa}
C.~Pa (Zhengguo Bai), 
A new definition of the Godeaux sequence of quadrics, 
\emph{Amer. J. Math.} \textbf{69} (1947), 117--120.



\bibitem{Pa1956}
C.~Pa (Zhengguo Bai), 
On the surfaces whose asymptotic curves of one system are projectively equivalent,
\emph{Univ. Nac. Tucum{\'a}n Revista} \textbf{A3} (1942), 341--349.


\bibitem{Pick}
G.~Pick, 
{\"U}ber affine Geometie IV. 
Differentialinvarianten der Fl{\"a}chen gegen{\"u}ber affinen 
Transformationen,  
\emph{Berichte Verh Ges. Wiss. Leibzig} \textbf{69} (1917), 107--136.
 
\bibitem{SasakiRokko}
T.~Sasaki, 
\emph{Projective Differential Geometry and 
Linear Homogenous 
Differential Equations}, 
Rokko Lectures in Math. \textbf{5} (1999), Kobe Univ.


\bibitem{Sasaki}
T.~Sasaki, 
Line congruence and transformation of projective surfaces, 
\emph{Kyushu J. Math.} \textbf{60} (2006), 101--243.

\bibitem{Su1936}
B.~Su, 
On certain periodic sequence of Laplace of period 
four in ordinary space,
\emph{Sci. Rep. Tohoku Imperial Univ.} \textbf{25} (1936), 227--256.

\bibitem{Su1964}
B.~Su, 
On certain couples of closed Laplace sequences of period four in ordinary space, 
\emph{Chinese Math.}  \textbf{5} (1964), 151--174.

\bibitem{Su1983}
B.~Su, 
The growth and development of differential geometry in China. (Japanese)
\emph{S{\=u}gaku} \textbf{35} (1983), no.~3, 221--228.


\bibitem{Th}
G.~Thomsen, 
Sulle superficie minime proiettive, 
\emph{Annali Mat.} \textbf{5} (1928), 169--184

\bibitem{Wil1}
E.~J.~Wilczynski, 
\emph{Projective Differential Geometry of Curves 
and Ruled Surfaces}, Teubner, 1906.


\end{thebibliography}
\end{document}